\documentclass[letterpaper,10pt]{amsart}
\usepackage{indentfirst} 
\usepackage{amssymb}
\usepackage{amsmath}
\usepackage{mathabx} 
\usepackage{amsthm}
\usepackage{thmtools}
\usepackage{bbm}             
\usepackage[usenames,dvipsnames]{xcolor} 
\usepackage{MnSymbol}
\usepackage[backref=page,colorlinks=true]{hyperref} 

\usepackage{tikz}
\usetikzlibrary{arrows}
\usetikzlibrary{decorations.pathreplacing}

\usetikzlibrary{arrows}
\newcommand{\midarrow}{\tikz \draw[-triangle 90 reversed] (0,0) -- +(.1,0);}
\newcommand{\midarrowdown}{\tikz \draw[-triangle 90] (0,0) -- +(.1,0);}

\theoremstyle{plain}
\newtheorem{theorem}{Theorem}[section]
\newtheorem{remark}[theorem]{Remark}
\newtheorem{proposition}[theorem]{Proposition}
\newtheorem{lemma}[theorem]{Lemma}
\newtheorem{corollary}[theorem]{Corollary}

\theoremstyle{definition}
\newtheorem{definition}[theorem]{Definition}
\newtheorem{conjecture}[theorem]{Conjecture}
\newtheorem{example}[theorem]{Example}

\newcommand{\M}{\mathcal{M}}
\newcommand{\R}{\mathbb{R}}
\newcommand{\Z}{\mathbb{Z}}
\newcommand{\N}{\mathbb{N}}

\def\eps{\varepsilon}
\def\phi{\varphi}
\def\R{{\mathbb R}}

\def\N{{\mathbb N}}
\def\Z{{\mathbb Z}}

\def\F{{\mathcal F}}

\def\M{{\mathcal M}}

\def\es{{\emptyset}}
\def\sm{\setminus}

\def\diam{\mbox{\rm diam } }

\def\bd{\partial }
\def\le{\leqslant}
\def\ge{\geqslant}

\def\F{\mathcal{F}}
\def\M{\mathcal{M}}

\usepackage{xcolor}
\hypersetup{
    colorlinks,
    linkcolor={red!50!black},
    citecolor={blue!50!black},
    urlcolor={blue!80!black}
}
\begin{document}

\title{Differentiability of the pressure in non-compact spaces}
\date{\today}

\subjclass[2010]{37D35, 28D20, 46T20}
\keywords{Topological pressure, countable Markov shifts, compactifications}

\begin{thanks}
{We would like to thank Natalia Jurga, Paulo Varandas and An\'ibal Velozo for many interesting comments and suggestions. Also, we thank the referee for a careful reading and many comments and suggestions that improved the article.  G.I.\ was partially supported  by Proyecto Fondecyt 1190194 and by CONICYT PIA ACT172001}
\end{thanks}

\author[G.~Iommi]{Godofredo Iommi}
\address{Facultad de Matem\'aticas,
Pontificia Universidad Cat\'olica de Chile (UC), Avenida Vicu\~na Mackenna 4860, Santiago, Chile}
\email{\href{giommi@mat.uc.cl}{giommi@mat.uc.cl}}
\urladdr{\url{http://http://www.mat.uc.cl/~giommi/}}

\author[M.~Todd]{Mike Todd}
\address{Mathematical Institute\\
University of St Andrews\\
North Haugh\\
St Andrews\\
KY16 9SS\\
Scotland} \email{\href{mailto:m.todd@st-andrews.ac.uk}{m.todd@st-andrews.ac.uk}}
\urladdr{\url{http://www.mcs.st-and.ac.uk/~miket/}}

\begin{abstract}
Regularity properties of the pressure are related to phase transitions. In this article we study thermodynamic formalism for systems defined in non-compact phase spaces, our main focus being countable Markov shifts. We produce metric compactifications of the space which allow us to prove that the pressure is differentiable on a residual set and outside an Aronszajn null set in the space of uniformly continuous functions. We establish a criterion, the so called \emph{sectorially arranged} property, which implies that the pressure in the original system and in the compactification coincide. Examples showing that the compactifications can have rich boundaries, for example a Cantor set, are provided.
\end{abstract}

\maketitle

\section{Introduction}

Beginning with the work of Gibbs, the formalism of equilibrium statistical mechanics was developed to address questions and problems related to systems consisting of a large number of particles. During the early 1970s,  Ruelle and Sinai among others \cite{Dobr, ru, Sinai}, realised that the underlying mathematical structure of this formalism could be successfully applied in the dynamical systems setting. The monograph of Bowen \cite{bow} is a remarkable example of how well the formalism is fitted to solving difficult questions in uniformly hyperbolic dynamics. The main object of the theory is the \emph{pressure}. This is a functional,  related to a dynamical system $T:X \to X$, defined on some subsets of the space of  real  continuous functions. One of the main problems in equilibrium statistical mechanics is that of understanding phase transitions. In the mathematical context this is related to regularity properties of the pressure. If $T$ is a continuous map and $X$ a compact space, several authors  have studied this problem \cite{ip,ru,Wal92}. For example, Walters \cite{Wal92} proved that for systems with upper semi-continuous entropy map the lack of differentiability of the pressure is related to the non-uniqueness of equilibrium measures.  As Walters showed, Gateaux differentiability of the pressure is related to the concept of a \emph{tangent functional}, which in good settings coincide with equilibrium measures.  In Section~\ref{ssec:tangnoteq} we give an example of a tangent functional which is not an equilibrium measure.

 As observed in \cite[Section 3]{ip}, it follows from a result of Mazur that the pressure is Gateaux differentiable in a residual set of the space of continuous functions, see Proposition \ref{thm_topo}. For uniformly hyperbolic systems, Ruelle \cite{ru} proved that when restricted to H\"older functions the pressure is differentiable at every point.  We show (see Proposition \ref{thm:bl}) that in the class of continuous functions, the set at which the pressure is not differentiable is also Aronszajn null.  In Section~\ref{ssec:SFT} we consider the specific case of sub-shifts of finite type. This is a fundamental example since  uniformly hyperbolic systems can be coded with them and thus the thermodynamic formalism for sub-shifts of finite type  can be transferred  to uniformly hyperbolic systems, see \cite{bow}.  It also gives us a natural setting to demonstrate when all the standard theory goes through.

If the space $X$ is no longer assumed to be compact the situation is more complicated and the theory only mildly developed. In
Section \ref{sec:tnc} we describe several approaches to define the pressure and outline the difficulties in each case. Our aim is to describe the regularity properties of the pressure for a large class of continuous functions in this non-compact setting. We will concentrate on a particular type of system namely, countable Markov shifts (CMS). These can be thought of as non-compact generalisations of sub-shifts of finite type, which are defined by means of a countable directed graph. It turns out that these systems are symbolic models for a wide range of dynamical systems. Indeed, after the work  of Sarig \cite{s3},  countable Markov partitions have been constructed for a large class of dynamical systems. This allows for the construction of a semi-conjugacy between a relevant part of the dynamics with a CMS. This has been achieved in the following contexts: positive entropy $C^{\infty}$ diffeomorphisms in manifolds,  Sinai and Bunimovich billiards and interval maps with critical points and discontinuities,  to name a few (see the survey \cite{l} for more details).

Mauldin and Urba\'nski \cite{mu} and Sarig \cite{Sar99} defined pressure for certain classes of regular continuous functions (e.g. summable variations) $\Sigma \mapsto \R$ in the context of CMS $\sigma: \Sigma \to \Sigma$. We consider a variational definition of the pressure that holds for any continuous function. Our strategy to prove regularity results for the pressure, similar to those that hold in the compact case, see Section \ref{sec:tc}, is to construct a metric $d$  which gives compact completion $\bar\Sigma$ (which we then also refer to as the compactification), and to consider uniformly continuous functions $\phi\in UC_d(\Sigma)$. Then to prove that the pressure on the original system coincides with that of the compactification. In this way we can transfer the results from the compact setting to the non-compact one. Note that such functions must be bounded, so it only really makes sense to consider cases in which $h_{top}(\sigma)<\infty$.  Since the space of bounded uniformly continuous functions, as well as the boundary of  $\Sigma$, depend upon the metric, finding metrics for which we can apply this strategy and for which the set of uniformly continuous functions is as large as possible, is of interest. For a Markov graph with vertices $V$ we start with a metric $\rho$ on $V$ which then induces a metric on $\Sigma$, see \eqref{eq:metric}.  If $\rho$ is totally bounded then the induced metric on $\Sigma$ has compact completion $\bar \Sigma$.  We define the notion of $V$ being \emph{sectorially arranged}, which generalises the standard conditions on $\rho$ in this setting (type 2 in \cite{gs}).  Denote by $P_\Sigma(\cdot) , P_{\bar\Sigma}(\cdot)$ the pressures defined on $(\Sigma, \sigma)$ and its completion $(\bar\Sigma, \bar\sigma)$, respectively. We obtain the following result.

\begin{theorem}
Let $(\Sigma ,\sigma)$ be a finite entropy topologically mixing CMS, $\rho$ totally bounded and $d$ the corresponding metric in $\Sigma$. If  $V$ is sectorially arranged  and $\phi \in UC_d(\Sigma)$, then 
\begin{equation*}
P_\Sigma(\phi)=P_{\bar\Sigma}(\phi).
\end{equation*}
\label{thm:CMS_main}
\end{theorem}
This result readily yields a description of the regularity properties of $P_{\Sigma}$. Indeed, the pressure is differentiable in a residual set and outside an Aronszajn null set of the space of uniformly continuous functions,  see Corollary \ref{cor:CMS gat}.  We describe  the structure of the  resulting compactifications, see Section~\ref{ssec:struc}. Several examples exhibiting different boundaries are provided. In particular, we produce an example for which the boundary is a Cantor set, see Section~\ref{ssec:comp}.   We also provide examples of  metrics  in $\Sigma$ so that  the set $UC_d(\Sigma)$ is large.  As noted above, our final example in Section~\ref{ssec:tangnoteq} is a natural setting where there is a tangent functional for the compactified system, but no equilibrium measure.

Gurevich \cite{gu2}, Walters \cite{Wal78}, Zargaryan \cite{Zar86} and also Gurevich and Savchenko \cite{gs} explored this compactification approach to define the pressure. In \cite{gu2, gs, Zar86} c functions that depend only on finitely many coordinates were  considered. In this paper we extend those results to continuous functions and to a larger class of  metrics. Walters  \cite{Wal78} studied a general case in which  functions are assumed to satisfy some forms of \emph{dynamical  continuity} \cite[p.149]{Wal78}. Other compactifications of CMS have been considered. Fiebig and Fiebig \cite{ff1}  construct compactifications of locally compact systems that are larger than the one point compactification. This work was continued in \cite{f}, where is it shown that non conjugated systems can have a conjugated compactification. Schwartz  \cite{sh} extended the notion of Martin boundary to locally compact CMS and was able to obtain results related to the corresponding transfer operator. He proved the existence of an eigenfunction corresponding to a function of summable variations. This compactification, however, depends upon the function  and therefore changes with it. Thus, this approach does not seem well suited to obtain the differentiability results we are interested in.
Also note that our results do not assume  the system to be  locally compact, nor that the function is of summable variations.

To fix notation we denote by $C(X)$ the space of real continued functions endowed with the supremum norm $\| \cdot \|$ and by $C(X)^*$ the dual space. We say that a sequence of probability measures  $(\mu_n)_n$ on the Borel space $X$ converges in the weak* topology to a probability measure $\mu$ if for every $f:X \to \R$ continuous and bounded we have  $\lim_{n \to \infty} \int f d \mu_n = \int f d \mu$.

\section{Differentiability of the pressure in the compact case} \label{sec:tc}
In this section the dynamical systems considered are continuous maps defined on compact spaces. We define the pressure and describe in detail its differentiability properties, both from the topological and measure theoretic point of view: the former is relatively well known in the field, but the latter statements are new in this context. The particular case of sub-shifts of finite type is then given as a standard application.

\subsection{Thermodynamic formalism in compact metric spaces} \label{ssec:def}
Let $(X, d)$ be a compact metric space, $T:X \to X$ a continuous map and $\phi \in C(X)$. Given $\epsilon >0$ and $n \in \N$, we say that a set $E \subset X$ is $(n,\epsilon)-$\emph{separated} if, given $x,y \in E$, there exists $j \in \{0,1, \dots, n-1\}$ such that
$d(T^j(x), T^j(y))> \epsilon$. Let
\begin{equation*}
Q_n(T,\phi, \epsilon)= \sup \left\{\sum_{x \in E} e^{\sum_{i=0}^{n-1} \phi(T^i x)} : E \subset X \text{ is } (n,\epsilon)-\text{separated}	\right\}.
\end{equation*}
Let $Q(T,\phi, \epsilon)=\limsup_{n\to \infty} \frac{1}{n} \log Q_n(T,\phi, \epsilon)$. The \emph{pressure} of $T$ is the map $P:C(X) \to \R \cup\{\infty\}$ defined by $P(\phi)=\lim_{\epsilon \to 0} Q(T,\phi, \epsilon)$ (see  \cite[Chapter 9]{Wal82} for details). Denote by $\M_T$ the space of $T$-invariant probability measures endowed with the weak* topology. The pressure satisfies the following properties \cite[Chapter 9]{Wal82}.

\begin{proposition}\label{prop:pre}
Let $T:X \to X$ be a continuous map on the compact metric space $X$ and $\phi \in C(X)$.
\begin{enumerate}
\item If $c \in \R$ and $\phi \in C(X)$ then $P(\phi + c)= P(\phi) +c$.
\item If $\phi, \psi  \in C(X)$ satisfy $\phi(x) \leq \psi (x)$ then $P(\phi) \leq P(\psi)$.
\item The function $P$ is Lipschitz continuous.
\item The function $P$ is convex.
\item \label{vp} $P(\phi)=\left\{h(\mu) + \int \phi \ d \mu : \mu \in \M_T\right\}$, where $h(\mu)$ denotes the entropy of $\mu$.
\end{enumerate}
\end{proposition}

\begin{remark}
Proposition \ref{prop:pre} item \eqref{vp} shows that, if $X$ is compact,  the pressure does not depend on the metric, as long as it generates the Borel $\sigma-$algebra of $X$. See  \cite[p.171]{Wal82} for a related discussion.
\end{remark}

The topological entropy of $T$ is defied as $h_{top}(T)=P(0)$. A measure $\mu \in \M_T$ such that $P(\phi)=h(\mu) + \int \phi \ d \mu$ is called \emph{equilibrium measure} for $\phi$.

\begin{remark}
Note that an equivalent definition of pressure can be given using open covers instead of $(n,\epsilon)-$separated sets (see \cite[Chapter 9]{{Wal82}}). Approaches using convex analysis to define the pressure have been used in \cite[Section 2]{ip} and \cite{bcmv}. 
\end{remark}

\subsection{Gateaux differentiability and equilibrium measures}
In this subsection we consider the regularity properties of the pressure considering a weak form of differentiability for functionals on Banach spaces.

\begin{definition}
The pressure map $P:C(X) \to \R$ is said to be \emph{Gateaux differentiable} at $\phi \in C(X)$ if for every $\psi  \in C(X)$ the following limit exists
\begin{equation*}
\lim_{t \to 0} \frac{P(\phi+t\psi) - P(\phi)}{t}.
\end{equation*}
\end{definition}

\begin{definition} Let $X$ be compact metric space, $T:X \to X$  a continuous map  of finite entropy and $\phi \in C(X)$. A measure $\mu \in \M_T$ is called a \emph{tangent functional} to $P$ at $\phi$ if  for every $\psi  \in C(X)$ we have that
\begin{equation*}
P(\phi+\psi) - P(\phi) \geq \int \psi d \mu.
\end{equation*}
Denote by $t_{\phi}(T)$ the collection of tangent functionals to $P$ at $\phi$.
\end{definition}

Note that for every $\phi \in C(X)$ the set  $t_{\phi}(T)$ is non empty and convex. Moreover, if $\mu$ is an equilibrium measure for $\phi$ then $\mu \in t_{\phi}(T)$ (see \cite[p. 225]{Wal82}). The following results obtained by Walters characterise Gateaux differentiability of the pressure in terms of tangent functionals (see \cite[Corollary 2 and Corollary 4]{Wal92}).

\begin{proposition} [Walters]  Let $X$ be compact metric space and $T:X \to X$  a continuous map  of finite entropy.
\begin{enumerate}
\item The pressure of $T$ is Gateaux differentiable at $\phi \in C(X)$ if and only if there is a unique tangent functional to $P$ at $\phi$.
\item The pressure of $T$ is Gateaux differentiable at $\phi \in C(X)$  if and only if there is a unique measure $\mu$ with the property that whenever $\mu_n\in \M_T$ satisfies
\begin{equation*}
\lim_{n \to \infty} \left( h(\mu_n) + \int \phi d\mu_n \right)= P(\phi),
\end{equation*}
then $\mu_n \to \mu$. In this case $\mu$ is the unique tangent functional.
\end{enumerate}
\label{prop:tang}
\end{proposition}
The relation between tangent functionals and equilibrium measures depends on the continuity properties of the entropy map as explained in the following result (see \cite[Theorem 5]{Wal92})

\begin{proposition} [Walters] \label{lem_eq}
Let $X$ be a compact metric space, $T:X \to X$  a continuous map  of finite entropy and $\phi \in C(X)$. 
A measure $\mu \in \M_T$ which is a tangent functional to $P$ at $\phi$ is not an equilibrium measure for $\phi$ if and only if the entropy map, $\nu \to h(\nu)$ defined in $\M_T$,  is not upper semi-continuous at $\mu$.
 \end{proposition}

In the next subsections we consider the problem of determining how large the set  at which the pressure is Gateaux differentiable is. We address this question from a topological and a measure theoretic of point of view.

\subsection{Gateaux differentiability from a topological perspective} Recall that a subset of a topological space is a $G_{\delta}$-set if it is a countable intersection of open sets and a dense $G_{\delta}$-set is called \emph{residual}. 
In 1933, Mazur proved that:  if $E$ is a separable Banach space and $F$ a continuous convex function defined on a convex open subset $D$ of $E$, then the set of point where $F$ is  Gateaux differentiable is a residual set in $U$ (see \cite[Theorem 1.20]{p3} and \cite[Section 3]{ip}). The following is a particular case of this result.

\begin{proposition}\label{thm_topo}
Let $X$ be a compact metric space and $T:X \to X$  a continuous map  of finite entropy. The set of points at which the pressure $P:C(X) \to \R$ is  Gateaux differentiable is a residual set in $C(X)$.
\end{proposition}
That is, the pressure is Gateaux differentiable in a large set from the topological point of view.

\begin{remark}
Mazur's Theorem holds in this setting because $X$ is compact and hence the Banach space $C(X)$ is separable. 
\end{remark}

\subsection{Gateaux differentiable from a measure theoretic perspective}

A classical result by Rademacher states that every Lipschitz map $F:\R^n \to \R^m$ is  Lebesgue almost everywhere differentiable. At least since the early 1970s, a great deal of work has been devoted to extend this result to Lipschitz maps between Banach spaces with respect to the Gateaux derivate. In order to do so, a notion of \emph{null set}  in Banach spaces is required. Several such notions  have been proposed, for example: cube null, Gauss null and Aronszajn null  (see \cite[Section 6]{bl} and references therein for definitions, properties and equivalences). 

Let $B$ be a separable Banach space. For each  $0 \neq a \in  B$ let $\mathcal{A}(a)$ be the family of Borel sets $A \subset B$ which intersect each line parallel to $a$ in a set of one-dimensional Lebesgue measure zero. That is, for every $x \in B$ we have
$\text{Leb}(\{t \in \R : x +ta \in A  \})=0$. If $(a_n)_n$ is a sequence of non zero elements in $B$, we denote by $\mathcal{A}((a_n))$ the collection of Borel sets $A$ such that $A=\bigcup A_n$, where $A_n \in \mathcal{A}(a_n)$ for every $n \in \N$. The following definition was proposed by Aronszajn, see \cite{a} and \cite[pp.141-142]{bl}.

\begin{definition}
A Borel set $A$ in a separable Banach space $B$ is called \emph{Aronszajn null} if $A$ belongs to $\bigcap \mathcal{A}((a_n))$, where the intersection is taken over  all sequences whose linear span is dense in $B$. 
\end{definition}

The class of Aronszajn null sets is closed under countable unions, hereditary class of Borel subsets of $B$, which does not contain open sets. Moreover, if $B$ is finite dimensional then it coincides with the family of Borel sets of zero Lebesgue measure. Note that Cs\"ornyei \cite[Theorem 1]{c} proved that in every separable Banach space the  class of Aronszajn null sets, Gauss null sets and cube null sets coincide. Aronszajn \cite[Main Theorem]{a} (see also \cite[Theorem 6.42]{bl}) extended Rademacher's Theorem to separable Banach spaces replacing the notion of zero Lebesgue measure with that of Aronszajn null.  Since the pressure is a  Lipschitz map from a separable Banach space to the real numbers, this result describes  its differentiability properties from a measure theoretic point of view.

\begin{proposition}\label{thm:bl}
Let $X$ be compact metric space and $T:X \to X$  a continuous map  of finite entropy and $U \subset C(X)$ be an open set.
The set of points at which the pressure $P:U \to \R$ is not Gateaux differentiable is Aronszajn null.
\end{proposition}

\subsection{Fr\'echet derivative} In this subsection we show that the stronger notion of Fr\'echet derivative is too strong for our purposes.  
\begin{definition}
Let $X$ be compact metric space, $T:X \to X$  a continuous map of finite entropy and $\phi \in C(X)$. 
The pressure $P:C(X) \to \R$ is Fr\'echet differentiable at $\phi$ if there exists $\Gamma \in C(X)^*$ such that
\begin{equation*}
\lim_{\psi \to 0} \frac{|P(\phi+\psi) -P(\phi) -\Gamma(\psi)|}{\| \psi \|}=0.
\end{equation*}
\end{definition}
If $P$ is Fr\'echet differentiable then it is Gateaux differentiable and in that case $\Gamma(\psi)= \int \psi \ d \mu_{\phi}$, where $\mu_{\phi}$ is the unique tangent functional at $\phi$.  A version of the reverse implication with stronger assumptions was obtained by Israel and Phelps (see \cite[p.144]{ip}).
\begin{proposition} \label{prop:gf}
If $P$ is Gateaux differentiable on an open set then it is Fr\'echet differentiable.
\end{proposition}
The following result was proved by Walters in \cite[Theorem 6 (vi)]{Wal92} (see also \cite{ip} and  \cite[Proposition 1]{dv}), it shows that the Fr\'echet derivative is not a well suited notion of derivative for dynamical systems having many invariant measures.
\begin{lemma}
If $P$ is Fr\'echet differentiable at $\phi$ then it is affine on a neighbourhood of $\phi$.
\end{lemma}
This follows from another characterisation of Fr\'echet differentiability  \cite[Theorem 6 (v)]{Wal92} namely, the pressure is Fr\'echet differentiable at $\phi$ if and only if there exists a unique equilibrium measure $\mu_{\phi}$ and
\begin{equation*}
P(\phi) > \sup \left\{h(\mu) + \int \phi d \mu :  \mu \in \M_T, \text{ ergodic and } \mu \neq \mu_{\phi}		\right\}.
\end{equation*}
That is, there is no sequence of ergodic measures $(\mu_n)$ in $\M_T$ different from $\mu_{\phi}$ such that
\begin{equation*}
\lim_{n \to \infty} \left(	h(\mu_n) + \int \phi d \mu_n	\right) =P(\phi).
\end{equation*}
In particular, if the set of ergodic measures is entropy dense then the pressure is not Fr\'echet differentiable at any point. See \cite[Theorem 2.1]{ps} for precise definitions and weak conditions which implies entropy denseness of ergodic measures.

\subsection{Sub-shifts of finite type} \label{ssec:SFT}
We conclude this section considering the particular, but important, case in which $(\Sigma, \sigma)$ is a sub-shift of finite type defined on a finite alphabet.   That is, let $N\ge 2$ and $A=(a_{i, j})_{i, j}$  a $N\times N$ matrix with entries in $\{0, 1\}$. The symbolic space is defined by
$$\Sigma= \left\{ \underline{x}=(x_0, x_1, \ldots): x_i \in\{1, \ldots, N\} \text{ and } a_{x_i, x_{i+1}}=1  \text{ for each } i\in \N_0  \right\}.$$ 
It is a compact space with the topology inherited from the product topology.  The function $d(\underline x,\underline y)$ defined by $1$ if  $x_0 \neq y_0$; equal to $2^{-k}$ if $x_i=y_i$ for $i \in \{0, \dots, k\}$ and $x_{k+1} \neq y_{k+1}$; and $0$ of $\underline x=\underline y$, is a metric on $\Sigma$, and thus induces the `cylinder topology' (see Section~\ref{sec:CMS} for more information). The dynamics is the left shift $\sigma:\Sigma\to \Sigma$, i.e., $\sigma (x_0, x_1, \ldots)= (x_1, x_2, \ldots)$.
We will  assume that this system is topologically mixing, which means that for each $i, j\in \{1, \ldots, N\}$ there is a finite collection $i=x_0, x_1, \ldots, x_\ell=j$ such that $A_{x_k, x_{k+1}}=1$ for $k=0, \ldots, \ell-1$. In this setting the entropy map is upper semi-continuous \cite[Theorem 8.2]{Wal82}. We have the following results, that are consequences of the more general statements of the previous sections.
\begin{proposition} Let $(\Sigma, \sigma)$ be a  topologically mixing sub-shift of finite type defined on a finite alphabet and $P:C(\Sigma) \to \R$ the  pressure.
\begin{enumerate}
\item The pressure $P$ is Gateaux differentiable in a $G_{\delta}$-set.
\item  The pressure $P$ is Gateaux differentiable  outside an  Aronszajn null set.
\item The pressure $P$ is Gateaux differentiable at $\phi$ if and only if there exists a unique equilibrium measure $\mu_{\phi}$ for $\phi$.
\item The pressure $P$ is not  Gateaux differentiable in a dense set.
\item The pressure $P$ is not Fr\'echet differentiable at any point.
\end{enumerate}
\end{proposition}

\begin{proof}
The first statement is Proposition \ref{thm_topo}, while  the second is Proposition \ref{thm:bl}. The third follows from the fact that the entropy map is upper semi-continuous and Proposition \ref{lem_eq}. The fact that the pressure in nowhere Fr\'echet differentiable was proved Walters in \cite[Corollary 9]{Wal92}. Finally, the fourth statement follows from the the fifth together with Proposition \ref{prop:gf}.
\end{proof}

\section{Thermodynamic formalism  on non-compact spaces} \label{sec:tnc}

While thermodynamic formalism is well developed for continuous maps defined on compact metric spaces, the situation is far less satisfactory if the compactness assumption is dropped. Indeed, if $T:X \to X$ is a continuous map and $X$  a non-compact metric space even the definition of pressure is a subtle matter. For example, the definition given in Section~\ref{ssec:def} based on the notion of $(n,\epsilon)$-separated sets depends upon the metric. That is, two different metrics generating the same topology can yield different values of the pressure. Explicit examples of this can be found in  \cite[Remark (15) p.171]{Wal82} or \cite[p.254]{hk}. If the notion of pressure is to satisfy the variational principle, $P(\phi)=\sup \left\{h(\mu) + \int \phi \ d \mu : \mu \in \M_T\right\}$ (Proposition \ref{prop:pre} item \eqref{vp}), then its value can only depend on the Borel structure on $X$ and not on the metric. Indeed, both the entropy of an invariant measure, the continuous functions and their integrals only depend on the Borel structure.

Other approaches, also based on the definition of pressure for a compact metric space, have been suggested (see, for example, \cite{gs}). The \emph{interior pressure} of $T$ at the continuous function $\phi$ is defined by
\begin{equation*}
P^{\text{int}}(\phi)= \sup \left\{P_{\Lambda}(\phi) : \emptyset \neq \Lambda \subset X \text{ compact and }T-\text{invariant}			\right\},
\end{equation*}
where $P_{\Lambda}(\phi)$ denotes the pressure of $\phi$ restricted to the set $\Lambda$. We assume that $\sup \emptyset =- \infty$.  This definition may work well for dynamical systems having a considerable number of compact invariant subsets. However, for systems lacking such sets the interior pressure may not satisfy the variational principle. The following example has been suggested several times (see see \cite{hk, gs}). Let $T:\hat{X} \to \hat{X}$ be a minimal system of positive entropy defined over the compact metric space $\hat{X}$. Let $X$ be  $\hat{X}$ minus an orbit.
Then, $T:X \to X$ is a continuous map defined over the non-compact metric space $X$ that has no $T-$invariant compact subsets. In particular, $P^{\text{int}}(0) < \sup \left\{h(\mu) : \mu \in \M_T\right\}= h_{top}(T)$. Examples of this type can be constructed in manifolds of any dimension (see \cite{bcl, Ree81}). Another drawback of the interior pressure is that, since the space of continuous functions $C(X)$ may not be separable, we cannot directly apply the results of the previous sections to describe its differentiable properties.

A different approach is to suppose that the set $X$ can be continuously embedded in a compact metric space $(\hat{X}, \hat{d})$ and that the continuous function $\phi: X \to \R$ can be continuously extended to $\hat{X}$ (we also denote by $\phi$ the extension). In this case the \emph{exterior pressure} is defined by
\begin{equation} \label{def:pext}
P^{\text{ext}}(\phi)= \inf \left\{P_{(\hat{X}, \hat d)}(\phi) :  (\hat{X}, \hat d)   \right\},
\end{equation}
where the infimum is taken over all possible embeddings  $(\hat{X}, \hat d)$.  One drawback of this approach is that the functions $\phi$ need to be bounded, so for systems with infinite topological entropy the pressure will always be infinite.

Walters \cite[Theorem 8]{Wal78} exploited this idea to extend the Ruelle Perron Frobenius theorem to some dynamical systems defined on non-compact spaces. With this approach he was able to study interval maps with countably many branches (the so called  $f$-maps) and a class of countable Markov shifts. He constructed equilibrium measures for a large class of function. Moreover,  he constructed functions with two equilibrium measures, recovering results by Hofbauer \cite{h}.

Bowen \cite{bow1} gave a definition of entropy for a system defined on a non-compact set that is contained in a compact metric space. Pesin and Pitskel' \cite{PesPit} further developed this approached and proposed a definition of pressure in the same setting. Examples with interesting properties can be constructed in this context. 

\begin{example}
We begin with an  example of dynamical system defined in non-closed set for which its entropy is positive and strictly smaller than that of its compactification. Let $T:X \to X$ be a continuous map defined on a metric space $X$ and $\mu$  a $T-$invariant probability measure. A point $x \in X$ is \emph{generic}   for $\mu$ if the sequence of empirical measures 
\begin{equation*}
\delta_{x,n}:= \frac{1}{n} \sum_{i=0}^{n-1} \delta_{T^i x}
\end{equation*}
converges in the weak* topology to $\mu$. Denote by $G(\mu)$ the set of generic points for $\mu$. Bowen \cite[Theorem 3]{bow1} showed that 
if $\mu$ is ergodic then $h(T|G(\mu))= h(\mu)$. Let $\sigma: \Sigma \to \Sigma$ be the full-shift on two symbols and $\mu$ a Gibbs measure different from the (unique) measure of maximal entropy. Then, the system $\sigma: G(\mu) \to G(\mu)$ has topological entropy
$h(\mu)< \log 2$ and it is topologically transitive. Since the measure $\mu$ gives positive mass to any open set in $\Sigma$ we have that  the closure of $G(\mu)$ is $\Sigma$. That is, its completion with respect to any metric compatible with the cylinder topology  has entropy strictly larger than the original system. In particular, the boundary supports the $(1/2, 1/2)$-Bernoulli measure. 
\end{example}

\begin{example} The following are examples of systems not satisfying the variational principle. Let $\sigma : \Sigma \to \Sigma$ be the full-shift on two symbols and $g:\Sigma \to \R$ a continuous function. The \emph{irregular set for the Birkhoff averages of $g$}  is defined by
\begin{equation*}
\mathbb{B}(g):=\left\{	x \in \Sigma : \lim_{n \to \infty} \frac{1}{n} \sum_{i=0}^{n-1} g(\sigma^ix)  \text{ does not exist}	\right\}.
\end{equation*}
The set $\mathbb{B}(g)$ is invariant but, in general, not compact. It does not support any $\sigma-$invariant measure.
Barreira and Schmeling \cite[Theorem 2.1]{bs} showed that if $g$ is not cohomologous to a constant then $h(\sigma | \mathbb{B}(g))= h_{top}(\sigma)= \log 2$. Moreover, the set  $\mathbb{B}(g)$ is dense in $\Sigma$. Therefore, the system $\sigma: \mathbb{B}(g) \to \mathbb{B}(g)$
does not satisfy the variational principle. Actually, even a smaller set has the same property. The following example appears in \cite[Proposition A.2.1]{pe}. Let
\begin{equation*}
Z:=  \left\{x \in \Sigma :  x \notin G(\mu) \text{ for any } \mu \in \M_{\sigma} \right\}.
\end{equation*}
The system $\sigma: Z \to Z$ has entropy equal to $\log 2$ and 
\begin{equation*}
0= \sup \left\{ h(\mu) : \mu \in \M_{(\sigma | Z)} \right\} < h(\sigma | Z)= \log 2.
\end{equation*}
We note that Thompson \cite{Tho11} proposed a definition of pressure in the same setting as \cite{bow1, PesPit} with a suitable variational principle.  In general this notion is larger than the definition of Pesin and Pitskel.
\end{example}

A problem with the definitions of Bowen and Pesin and Pitskel (and also with that of Thompson)  is that we need a compact reference space and there are plenty of natural examples for which such a space is not available.

Finally, we consider yet another definition of pressure in the case that $T:X \to X$ is a continuous map defined on a non-compact space.  The \emph{variational pressure} of the continuous function $\phi:X \to \R$ is defined by
\begin{equation*}
P_{var}(\phi)= \sup \left\{h(\mu) + \int \phi \ d \mu : \mu \in \M_T	\text{ such that } \int 	\phi \ d \mu > - \infty	\right\}.
\end{equation*}

\begin{proposition} \label{prop:var}
The variational pressure satisfies the following properties
\begin{enumerate}
\item If $c \in \R$ and $\phi \in C(X)$ then $P_{var}(\phi + c)= P_{var}(\phi) +c$.
\item If $\phi, \psi  \in C(X)$ satisfy $\phi(x) \leq \psi (x)$ then $P_{var}(\phi) \leq P_{var}(\psi)$.
\item If $\phi, \psi  \in C(X)$ satisfy $P(\phi)<\infty$, $P(\psi)<\infty$ and $\|\phi -\psi\| <\infty$ then $|P_{var}(\phi)- P_{var}(\psi)| \leq \| \phi - \psi\|$.
\item The function $P_{var}$ is convex.
\item If $X$ is a compact metric space then $P(\phi)=P_{var}(\phi)$.
\end{enumerate}
\end{proposition}

\begin{proof}
The first two claims are direct from the definition. For the third, by the first two properties we have
\begin{eqnarray*}
P_{var}(\psi) - \| \phi - \psi\| =P_{var}\left(P_{var}(\psi) - \| \phi - \psi\| \right) \leq P_{var}(\phi) \leq &\\ 
P_{var}\left(P_{var}(\psi) + \| \phi - \psi\| \right) =P_{var}(\psi) + \| \phi - \psi\|.
\end{eqnarray*}
That is $|P_{var}(\phi)- P_{var}(\psi)| \leq \| \phi - \psi\|$.

Given $\mu \in \M_T$  the map $\phi \mapsto \left( h(\mu) + \int \phi  \ d \mu \right)$ is affine.  Therefore, $P_{var}$ is convex being the  supremum of affine, hence convex, maps (see \cite[Theorem 5.5]{r}). The last claim is Proposition \ref{prop:pre}\eqref{vp}.
\end{proof}

\begin{remark}
If $T:X \to X$ is a continuous map and $\phi:X \to \R$ a continuous function then 
$P^{\text{int}}(\phi) \leq P_{var}(\phi) \leq  P^{\text{ext}}(\phi)$.
All the inequalities can be strict, (see \cite[Theorem 1.8]{hk} and the above discussion). 
\end{remark}

In the next sections we will be interested in the case in which the variational and the exterior pressure coincide. More precisely, we will construct metrics $d$ in $X$ such that the corresponding completion $\hat{X}$ is  compact and for every uniformly continuous function $\phi:X \to R$ we have  $P_{var}(\phi)= P_{(\hat{X}, d)}(\phi)$.  Denote by  $UC_d(X)$ the space of bounded uniformly continuous functions on $X$  with respect to the metric $d$.

\begin{theorem} 
Let $T:X \to X$ be a finite entropy map defined on the non-compact topological space $X$. Suppose $X$ that is densely embedded in a compact metric space $(\hat{X}, d)$ so that for every $\phi \in UC_d(X)$,
\begin{equation*}
P_{var}(\phi)= P_{(\hat{X}, d)}(\phi).
\end{equation*}
If $U \subset UC_d(X)$ is an open set then
\begin{enumerate}
\item the pressure $P_{var}: U \to \R$ is Gateaux differentiable in a $G_{\delta}$ subset of $U$;
\item the set of points at which the pressure $P_{var}:U \to \R$ is not Gateaux differentiable is an Aronszajn  null set.
 \end{enumerate}
 \label{thm:ncpct_general}
\end{theorem}
The proof of this result is an immediate consequence of Proposition  \ref{thm_topo} and Proposition \ref{thm:bl}.

\section{Countable Markov shifts: preliminaries and results}
\label{sec:CMS}
In this section we consider the particular case in which the dynamical system defined over a non-compact space is a \emph{countable Markov shift (CMS)} $(\Sigma, \sigma)$. Let $A=(a_{i,j})_{i, j}$ be an $\N\times \N$ transition matrix with entries in $\{0, 1\}$ and let
$$\Sigma=\left\{ \underline{x}=(x_0, x_1, \ldots)\in \N^{\N_0}:a_{x_n, x_{n+1}}=1\right\}.$$
Thus the system is defined by a directed graph structure on $\N$.  When we want to emphasise that $\N$ is being used just as a countable set of vertices, we may denote it by $V$.  On the other hand, often it will be useful to use the implicit indexing of these vertices which $\N$ brings.  Let $\sigma:\Sigma \to \Sigma$ be the left shift, we will always assume that  $(\Sigma, \sigma)$ is topologically mixing (the definition is the same as the one we gave in Section~\ref{ssec:SFT}).  We consider $\Sigma$ endowed with the topology generated by the cylinder sets $\{Z: Z\in \mathcal{Z}_n \text{ for some } n\in \N\}$ where the $n$-cylinder $Z\in \mathcal{Z}_n$ containing $x\in \Sigma$ is of the form $Z=\{(y_0, y_1, \ldots)\in \Sigma: y_i=x_i \text{ for } i=0, \ldots n-1\}$.  In general, $\Sigma$ is a non-compact space. Moreover, it is locally compact if and only if the row sum of the transition matrix $A$ is always finite. For CMS we will adopt the variational definition of pressure. Let $\phi:\Sigma \to \R$ be a continuous function, the pressure of $\sigma$ at $\phi$ is defined by
\begin{equation*}
P_{\Sigma}(\phi):=P_{var}(\phi)= \sup \left\{h(\mu) + \int \phi \ d \mu :  \mu \in \M_{\sigma}	\text{ such that } \int 	\phi \ d \mu > - \infty	 \right\}.\end{equation*}
We include the subscript $\Sigma$ to emphasise the space in which the dynamics is defined.

\begin{remark}
In a series of articles starting in 1999, Sarig developed a theory of thermodynamic formalism for general topologically mixing CMS. Similar results, for a narrower class of systems, were obtained earlier by Mauldin and Urba\'nski \cite{mu}. Both, Sarig  and Mauldin-Urba\'nski considered regular continuous functions. Indeed, for functions of summable variations (see \cite[p.556]{s2} for precise statements) Sarig defined a notion of pressure, the so-called \emph{Gurevich pressure}. This notion satisfies the variational principle and is well suited for  the use of all the transfer operator machinery. In particular, he proved that if $\phi: \Sigma \to \R$ is of summable variations then  $P^{\text{int}}(\phi) =P_{var}(\phi)$, see \cite[Theorem 2]{Sar99} and \cite[p.557]{s2}. This identity, for $\phi=0$, was obtained earlier by Gurevich \cite{gu1}. We stress that we do not assume the summable variation condition on the function $\phi$ in the definition of $P_{\Sigma}(\phi)$.
\end{remark}

\begin{remark} \label{rem:fe}
Recall that the entropy of $(\Sigma, \sigma)$ is defined by $h_{top}(\sigma)=P_{\Sigma}(0)$.  If there exists $N \in \N$ such that $(\Sigma, \sigma)$ has infinitely many periodic orbits of period $N$ then $h_{top}(\sigma)= \infty$. This directly follows from the definition of Gurevich entropy, see \cite{gu1, gu2, Sar99}.  
\end{remark}

\subsection{Metrics, compactifications and differentiability of the pressure}
In Theorem \ref{thm:ncpct_general}  a strategy to extend the results on differentiability of the pressure obtained in the compact setting to dynamical systems defined on non-compact spaces, is proposed. We now discuss its implementation in the context of CMS, $(\Sigma, \sigma)$. The idea is to consider a metric $d$ on $\Sigma$, generating the cylinder topology, such that its completion $\bar\Sigma$ is compact. The main property we require from the completion is that if $\phi \in UC_d$  then $P_\Sigma(\phi)=P_{\bar\Sigma}(\phi)$. Note that, abusing notation, we denote the (unique) extension of $\phi$ to the completion of $X$ also by $\phi$.

We first consider a class of metrics in $\Sigma$, maintaining  the notation used in \cite{gs,Zar86}.  This will be of the form: for $\theta\in (0, 1)$ and $\underline x=(x_0, x_1, \ldots), \ \underline  y= (y_0, y_1, \ldots)\in \Sigma$, 
\begin{equation}
d(\underline x,\underline  y)=d_{\rho, \theta}( \underline x, \underline  y)=\sum_{n\ge 0}\theta^n\rho(x_n, y_n)
\label{eq:metric}
\end{equation}
with $\rho: V \times V \to [0, 1]$ a metric on $V$. 
In this setup, $d_{\rho, \theta}$ generates the cylinder topology and $\sigma$ is uniformly continuous.  In particular $\sigma$ extends to the completion of $\Sigma$, which we denote by $\bar\Sigma_\rho$,
and $(\bar\Sigma_\rho, \sigma)$ is a continuous dynamical system. We denote by $\bar d$ the metric on  $\bar\Sigma_\rho$.  It will be important for us that this process yields a compact space.

We say $\rho$ is of \emph{vanishing type} if 
$$\lim_{n\to \infty}\sup_{i, j\ge n}\rho(i, j)=0,$$
and \emph{non-vanishing type} if not (observe that this condition is independent of the choice of enumeration of the vertices).  In \cite{gs}, vanishing type metrics were referred to as type 2 metrics, with type 1 metrics defined by $\inf_{i\neq j}\rho(i, j)>0.$  We will not be interested in type 1 metrics here since in the CMS setting they are not totally bounded and so the completion of $\Sigma$ in such a metric is not compact.

 To ensure compactness of the completion it is sufficient that $\rho$ be totally bounded, see below.
Given $\rho$ on $V$, we will always assume $d$ is given by \eqref{eq:metric}.

\begin{remark} The following observations will be important.
\begin{enumerate}
\item
If $\rho$ is totally bounded then the induced metric $d$ on $\Sigma$ as in \eqref{eq:metric} is totally bounded.  This can be seen by supposing $\eps>0$ and finding a finite cover of $V$ by $V_1, \ldots, V_k$ of diameter less than $\epsilon$.  Let $n\in \N$ have $\theta^n<\eps$.  Then the induced cylinders $[V_{i_1}, \ldots, V_{i_n}]\cap \Sigma$ for $V_{i_j}\in \{V_1, \ldots, V_k\}$ cover $\Sigma$.  These have cardinality at most $k^n$.  Therefore total boundedness of $\rho$ implies $(\bar\Sigma_\rho, \bar d_{ \rho, \theta})$ is compact.

\item Any metric of vanishing type is totally bounded.  A simple example of a non-vanishing type metric which is not totally bounded is: for $a, b\in \N$ let $\rho(a, b)=0$ if $a=b$ and $\rho(a, b)=1$ if $a\neq b$.  Given $d$ as in  \eqref{eq:metric}, the completion $\bar\Sigma_\rho$ is not compact.
\item
If $\rho$ is totally bounded then $\phi:\Sigma \to \R$ is uniformly continuous in $d$ if and only if it extends to a  function $\bar\phi:\bar\Sigma _{\rho, \theta}\to \R$ which is continuous in $\bar d$. 
\end{enumerate}
\end{remark}

 Let $\N_\rho$ denote the completion of $\N$ with respect to $\rho$ and $\bd \N_\rho=\N_\rho\sm \N$.  We denote the boundary of $\Sigma$ by $\bd \Sigma=\bd \Sigma_\rho$.  Sometimes it will be more convenient to emphasise that $\N$ is the set of vertices for our shift and denote these by $V$, using the notation for boundaries  as above.

This metric compactification approach to study CMS began with the work of Gurevich in 1970 \cite{gu2} and Walters in 1978 \cite[Section 3.3]{Wal78}. It was further developed by Zargaryan \cite{Zar86}, who considered  $\rho(a, b)=\left|\frac1a-\frac1b\right|$ if $a, b\in \N$ and $a\neq b$ and $\rho(a, a)=0$, thus $\rho$ is totally bounded. The main result in \cite{Zar86} is that for the corresponding metric $d$ in $\Sigma$ and $\phi:\Sigma \to \R$ uniformly continuous, we have
$P_{\bar\Sigma}(\phi)= P^{\text{int}}(\phi)$.  As pointed out by Gurevich and Savchenko \cite[\textsection 1]{gs}, the precise form of the metric is not important for this result, so long as it is of vanishing type.  The completion with respect to such a metric can be understood as a particular shift on the one point compactification of the alphabet (usefully thought of as $\{\ldots, \frac13, \frac12, 1\}$): namely $\bar\N_\rho=\N\cup\{\infty\}$.   As we will show later, $(\bar\Sigma, \sigma)$ is often not Markov (see Sections~\ref{ssec:comp} and \ref{ssec:nsec}), and the transitions to and from $\bd \Sigma$ depend on the structure of $\Sigma$ not the specific type of $\rho$ (see Section~\ref{ssec:struc}).

For many systems, particularly locally compact ones, metrics $d_{\rho, \theta}$ with $\rho$ of vanishing type are very natural, indeed the one point compactification of $\Sigma$ coincides with $\bar\Sigma$.  However, there are many cases where $\rho$ being non-vanishing, which enriches the boundary $\bd\Sigma$ and enlarges $UC_d(\Sigma)$, is natural.  The following result always holds for these metrics.

\begin{lemma}
If $\phi\in UC_{d_{\theta, \rho}}(\Sigma)$ then $P^{\rm{int}}(\phi)=P_{\rm{var}}(\phi)$.
\end{lemma}

\begin{proof}
This follows from the equivalent result for potentials of summable variation in \cite[Theorem 2]{Sar99} (in the vanishing case we could use \cite[Theorem 1.5]{gs}).  As in the proof of Proposition~\ref{prop:var}, if $\phi, \psi:\Sigma\to \R$ are continuous then $\|\phi-\psi\|\le \eps$ implies $|P^{\text{int}}(\phi)- P^{\text{int}}(\psi)|, |P_{\text{var}}(\phi)- P_{\text{var}}(\psi)|<\eps$, so it is sufficient to show that potentials of summable variation are dense in $UC_{d_{\theta, \rho}}(\Sigma)$.  Recall that the $n$-th variation of $\phi:\Sigma\to \R$ is $V_n(\phi):=\sup_{Z\in \mathcal{Z}_n}\sup_{\underline x, \underline y\in Z}|\phi(\underline x)- \phi(\underline y)|$, and the potential has summable variations if $\sum_{n\ge 1}V_n(\phi)<\infty$.

Suppose $\phi\in UC_{d_{\theta, \rho}}(\Sigma)$ and let $\eps>0$.  By uniform continuity there exists $\delta>0$ such that $d(\underline x, \underline y)<\delta$ implies $|\phi(\underline x)- \phi(\underline y)|<\eps$.  Choose $n\in \N_0$ such that $\theta^n/(1-\theta)<\delta$ and observe that from \eqref{eq:metric}, any $Z\in \mathcal{Z}_n$ has $\diam Z<\delta$.  So we can pick an arbitrary point $\underline x_{Z}\in Z$ and define, for $\underline x\in Z$, the function $\tilde\phi:\Sigma\to \R$ by $\tilde \phi(\underline x)= \phi(\underline x_{Z})$.  Then $\tilde \phi$ is of summable variations (indeed it is locally constant) and $\|\phi-\tilde\phi\|<\eps$, as required.  
\end{proof}

We next focus on a particular kind of metric structure, where the vertices $V=\N$ can be `sectorially arranged' with respect to $\rho$, see Section~\ref{sec:sectors} for a definition.  The main result here is Theorem~\ref{thm:CMS_main}.
We will show that vanishing type metrics are sectorially arranged and hence, since there are also non-vanishing examples which are sectorially arranged, this theorem is a generalisation of \cite{Zar86}.  

Functions $\phi \in UC_d(\Sigma)$  extend to continuous functions on $\bar\Sigma$ and are therefore bounded.  Since we are interested in pressure, it therefore only makes sense in this context to look at systems with finite topological entropy.
We are able to state the following, which is a corollary of Proposition~\ref{thm:ncpct_general} and Theorem~\ref{thm:CMS_main}.

\begin{corollary}  \label{cor:CMS gat}
Let $(\Sigma ,\sigma)$ be a finite entropy CMS, $\rho$ be totally bounded and $d$ the corresponding metric in $\Sigma$. If  $V$ is sectorially arranged  and $U \subset UC_d(\Sigma)$ is an open set then
\begin{enumerate}
\item the pressure $P_\Sigma:U \to \R$ is Gateaux differentiable in a $G_{\delta}$ set of $U$;
\item the set of points at which the pressure $P_\Sigma:U \to \R$ is not Gateaux differentiable is an Aronszajn null set.
 \end{enumerate}
\end{corollary}

The properties established in the following definition trivially imply the conclusion of Theorem~\ref{thm:CMS_main} and Corollary~\ref{cor:CMS gat}.

\begin{definition}
Let $(\Sigma ,\sigma)$ be a finite entropy CMS. We say that it is \emph{interior rich} for a totally bounded $\rho$ if for any $\mu\in \M_{\bd\Sigma_\rho} , \phi\in C(\bar\Sigma, \R)$ and $\eps>0$ there is $\mu'\in \M_{\Sigma}$ such that
$$h(\mu')> h(\mu)-\eps \text{ and } \int\phi~d\mu'> \int\phi~d\mu-\eps.$$
\end{definition}

In fact, we will show that sectorially arranged examples are interior rich (see Remark \ref{rem:ir}), but will also give examples showing that the latter class is larger than the former (see Section \ref{ssec:nsec}).

\subsection{Preliminary structure} \label{ssec:struc}

In this subsection we give some idea of the structure of $\bd \Sigma$ which does not depend strongly on the form of $\rho$. It will be more convenient to write $\N$ rather than $V$ for our vertices.   We first show that $\bd\Sigma$ inherits some structure from $\Sigma$.

\begin{lemma}
We can write points in $\bd\Sigma$ as sequences $(x_0, x_1, \ldots)$ where $x_i\in \N\cup\bd\N_\rho$.
\end{lemma}

\begin{proof}
Suppose $x\in \bd\Sigma$ and that $(y^k)_k, (z_k)^k$ are Cauchy sequences in $\Sigma$ converging to $x$.  Suppose that $I$, a bounded subset of $\N_0$, and $\eps>0$ are such that for all large $k$, $d(y_i^k, \bd\N_\rho)>\eps$ whenever $i\in I$.  Then the same must be true of the $z_i^k$ for large $k$.  Hence $x$ can be represented at indices $I$ by the corresponding entry $y_i^k$.  On the other hand, if $I'$, a bounded subset of $\N_0$, is such that $d(y_i^k, \bd\N_\rho)\to 0$ as $k\to \infty$ for $i\in  I'$ then the same must be true for $(z_i^k)_k$, and the corresponding entries for $x$ can be taken in $\bd\N_\rho$.
\end{proof}

We write $i \to j$ for $i, j\in \N_\rho$ if there is a sequence $x=(x_0, x_1, \ldots)\in \bar\Sigma$ with $x_0=i$ and $x_1=j$.  Given two sets of vertices $A$ and $B$ we write $A\to B$ if there are $i\in A, j\in B$ such that $i\to j$.  We say that a set of vertices is \emph{connected} if for each two vertices in the set there is a directed path between them (this may only be in one direction). If $\rho$ is of vanishing type then $\bd\N_\rho$ is a single point which we denote by $\{\infty\}$, but in general $\bd\N_\rho$ may not even be countable.  The following lemma deals with a similar  setting to \cite[Proposition 5.1]{gs}.  Observe that nothing is assumed on the metric $\rho$.

\begin{lemma} Let $(\Sigma ,\sigma)$ be a finite entropy topologically mixing  CMS.
\begin{enumerate}
\item  
There is no element of the form $(x_0, x_1, \ldots)\in \bar\Sigma$ with $x_n=i\in \N$,  the points $x_{n+1}, \ldots,  
x_{n+m} \in \bd\N_\rho$ and $x_{n+m+1}=j\in \N$.
\item 
$\bd\bar\Sigma \to \bd\bar\Sigma$ is always allowed.
\item 
Any ergodic measure $\mu$ on $(\bar\Sigma,  \bar\sigma)$ for which $\mu(\Sigma)<1$ has $\mu(\bd\Sigma)= 1$.
 \end{enumerate}
 \label{lem:struc}
\end{lemma}

\begin{proof}
Suppose that in fact there is such a point $(x_0, x_1, \ldots)\in \bar\Sigma$ with $x_n=i\in \N$, $x_{n+1}, \ldots, x_{n+m}\in \partial \N$ and $x_{n+m+1}=j\in \N$.  Then there must exist a Cauchy sequence $(y^k)_k$ in $ \Sigma$, with $y_n^k=i$ and $y_{n+m+1}^k= j$, which converges to $x$.  Since, by topological mixing there is a path of length $\ell\in \N$ from $j$ to $i$, this and the Markov property imply that there must be infinitely many periodic points of period $m+1+\ell$ in $[i]$, which contradicts the finite entropy assumption (see Remark \ref{rem:fe}).

For the second part, suppose that $\bd\Sigma \to \bd\Sigma$ is not allowed.  Then there must exist $N\in \N$ such that whenever $n, n'>N$ then $n\to n'$ is forbidden.  By the pigeonhole principle there must exist $i, j\le N$ such that $i$ has infinitely many outgoing arrows and $j$ has infinitely many incoming arrows.  Then as in the first part we can find infinitely many periodic points of a finite period, contradicting the finite entropy assumption.

For the third part, note that it follows from the previous parts that 
$$\bd\Sigma=\left\{x=(x_0, x_1, \ldots)\in \bar\Sigma: \exists k\in \N_0 \text{ s.t. } x_{k+n}\in \bd \N_\rho \text{ for all } k\ge 0\right\}.$$
Therefore, if an ergodic $\bar\sigma$-invariant probability measure $\mu$ is such that $\mu(\partial \Sigma)>0$ then $\mu(\partial \Sigma)=1$.
\end{proof}

We provide examples exhibiting the allowed structure established in  Lemma~\ref{lem:struc}. It suffices to consider any vanishing type metric and to recall that for these metrics we can write $\bd \N_\rho=\{\infty\}$. We construct examples where: $n \to \infty$ for some $n\in \N$; no $n\in \N$ with $n \to \infty$; $\infty\to n$ for some $n\in \N$; and no $n\in \N$ with $\infty\to n$.  Note that in these three examples $h_{top}(\sigma)=\log 2$ and moreover we can join them together to obtain mixed behaviour.

\begin{example}[Renewal Shift: $1\to \infty$ and $\infty\nrightarrow n$]
For the renewal shift, the transition matrix $A=(a_{i,j})$ has $a_{1,j}=1$ for all $j\in \N$, $a_{i, i-1}=1$ for all $i\ge 2$ and  $a_{i, j}=0$ otherwise. Here we can define a Cauchy sequence $((1, n, n-1, \ldots))_n$, i.e., we specify the first two terms of the $n$th term in sequence as 1 and $n$, the next $n-1$ terms are then determined and the remainder can be chosen abitrarily.  So in the completion we must have $1\to \infty$.  Observe that there is no finite symbol $n$ which can be reached from a large $n'$ in one step, so $\infty\nrightarrow n$.
\label{eg:renew}
\end{example}

\begin{example}[Backwards Renewal Shift: $n\nrightarrow \infty$ and $\infty\rightarrow n$]
In this case the transition matrix  $A=(a_{i,j})$ has $a_{i,1}=1$ for all $i\in \N$, $a_{i, i+1}=1$ for all $i\in \N$ and  $a_{i, j}=0$ otherwise.  In this case we can define a Cauchy sequence $((n, 1, \ldots))_n$, so $\infty\rightarrow 1$ is allowed.  On the other hand there is no finite symbol $n$ which can reach larger and larger symbols in one step, so $n\nrightarrow \infty$.
\label{eg:brenew}
\end{example}

\begin{example}[One-sided random walk: $n\nrightarrow \infty$ and $\infty\nrightarrow n$]
Here the transition matrix  $A=(a_{i,j})$ has $a_{1,j}=1$ for $j=1, 2$ and $a_{i,i\pm 1}=1$ for all $i\ge 2$.  Clearly the only way we can go from a symbol to a large symbol is if the first symbol is also large, so ultimately we cannot have $n\to \infty$ for $n\in \N$.  Similarly if an initial symbol is large then the next symbol must also be large, so $\infty\to n$ for some $n\in \N$ is not allowed.
\label{eg:1side}
\end{example}

\section{Sectors}  \label{sec:sectors}

In this section, given a metric $\rho$ on the space of vertices we define the notion of \emph{sectorially arranged}. We prove that it implies that the (variational) pressure $P_{\Sigma}$ coincides with the pressure on the compactification induced by the metric. This, as we already pointed out, yields a good description of the regularity properties of the pressure.

\begin{definition} \label{def:sector}
We say that the vertex set $V=\N$ is \emph{sectorially arranged}, with respect to the metric $\rho$ in $V$, if there are sequences $N_k\to \infty$, $\hat\delta_k\to 0$ and 
$$V= \{1, \ldots, N_k\}\cup\bigsqcup_{i=1}^{p_k} V_k^i$$ 
for $p_k\in \N\cup\{\infty\}$, where each \emph{sector} $V_k^i$ is infinite and connected with $\diam V_k^i<\hat\delta_k$; for a given $k$, $V_k^i$ are not connected to each other; and  for $k\ge 2$, $V_{k}^i\subset V_{k-1}^{i'}$.
\end{definition}

\begin{remark} From Definition \ref{def:sector} we have:
\begin{enumerate}
\item  Note that each $V_k^i$ may have non-empty intersection with $\{1, \ldots, N_k\}$.
\item  If  we have nested sequence of sectors $V_1^{i_1}\subset V_2^{i_2}\subset \cdots$ then $\bigcap_{k} \overline{V_k^{i_k}}$ is a single point in $\bd V$.  
\item Conversely, for $v_\infty\in \bd V$ and $k\in \N$ there exists at least one $i= i(k, v_\infty)$ such that $v_\infty\in \bd V_k^i$.
\end{enumerate}
\end{remark}

\begin{lemma}
If $V$ is sectorially arranged and $v_\infty, v_\infty'\in \bd V$ then $v_\infty \to v_\infty'$ if and only if $v_\infty= v_\infty'$.
\label{lem:no_ent}
\end{lemma}

\begin{proof}
Suppose that $v_\infty\in \bd V$.  Then for any $k, n\in \N$ there is a point $\underline x\in \Sigma$ with $n$ of its first symbols in $V_k^{i(k, v_\infty)}$.  But this also means that $\sigma\underline x\in \Sigma$ has $n-1$ of its first symbols in $V_k^{i(k, v_\infty)}$.  So both $\underline x$ and $\sigma\underline x$ are close to $v_\infty$, which implies that $v_\infty \to v_\infty$ is an allowed transition.  In particular $(v_\infty, v_\infty, \ldots)\in \bd \Sigma$.  

Now suppose that  $v_\infty \to v_\infty'$ is allowed for  $v_\infty, v_\infty'\in \bd V$.  This means that there must be a sequence $(\underline x^n)_n$ such that $x_0^n\to v_\infty$ and  $x_1^n\to v_\infty'$ in $\rho$.  Since $V$ is sectorially arranged, for all large $n$ the vertices $x_0^n$ and $x_1^n$ must lie in the same sector.  Hence they must accumulate on the same point, which implies $v_\infty= v_\infty'$.
\end{proof}

Note that in this proof, the fact that $v_\infty$ can be the limit of more than one nested sequence of sectors means that it is possible for the sequence $(\underline x^n)_n$ to jump between different nested sectors, but this does not change the result.

\begin{proof}[Proof of Theorem~\ref{thm:CMS_main}]
Note that by \cite[Corollary 9.10.1]{Wal82} the pressure of $\phi$ with respect to $(\bar\Sigma, \bar\sigma)$ can be computed as
\begin{equation*}
P_{\bar\Sigma}(\phi)= \sup \left\{ h(\nu) + \int \phi \ d \nu: \nu \in  \M_{\bar\Sigma} \text { and ergodic}		\right\}.
\end{equation*}
Therefore, we can restrict our attention to ergodic measures. The main idea here is that since by Lemma~\ref{lem:no_ent} there is no entropy on the boundary, and $\M_{\bar\Sigma}=\M_\Sigma\sqcup\M_{\bd\Sigma}$,  it is sufficient to approximate the integrals $\int\phi~d\mu$ for $\mu\in \M_{\bd \Sigma}$ by integrals  $\int\phi~d\nu$ for $\nu\in \M_{\Sigma}$.  The measures  $\nu$ will be equidistributions on periodic cycles which approximate the fixed points at the boundary.

For each $k \in \N$, topological mixing implies that there is a $M_k \in \N$ such that any two vertices in $\{1, \ldots, N_k\}$ can be connected in less than $M_k$ steps.  Therefore, given $v_\infty\in \bd V$,  for any $k\in \N$ and $n\in \N$ we can find a periodic point $\underline z$ which spends $n$ of its iterates in $V_k^{i(k, v_\infty)}$ and at most $M_k+2$ of its iterates in $\{1, \ldots, N_k\}$.  So if $\eps>0$ and $\hat\delta_k$ is chosen so that $d(\underline x, \underline y)<\delta_k$ implies $|\phi(\underline x)- \phi(\underline y)|<\eps$, this means $|\phi(\sigma^j\underline x)- \phi(\sigma^j\underline y)|<\eps$ for $n$ terms $j$.  That is, for measure $\mu_{\underline z}$ the equidistribution on the orbit of $\underline z$,

$$\left|\int \phi~d\mu_{ \underline z} -\phi(\underline x_v)\right|< \eps + \frac{M_k+2}n\|\phi\|_\infty$$
where $\underline x_v$ is the point $(v_\infty, v_\infty, \ldots)$. Thus, we can approximate any measure $\mu$ supported on $\bd \Sigma$.  
\end{proof}

\begin{remark} \label{rem:ir}
It follows from the proof that, in this setting, $V$ being sectorially arranged implies interior richness.
\end{remark}

\begin{remark}
The idea that a good understanding of the behaviour of countable Markov shifts at \emph{infinity} can shed light on the dynamical properties of the system has been recently fomalised in \cite{IomTodVel19}. 
The measure theoretic entropy at infinity of  $(\Sigma, \sigma)$ can be defined by
\begin{equation*}
h_\infty :=\sup_{(\mu_n)_n\to 0}\limsup_{n\to\infty}h_{\mu_n}(\sigma), \label{eq:mte}
\end{equation*}
where $(\mu_n)_n\to 0$ means that the sequence $(\mu_n)_n$ converges  on cylinders to the zero measure (this means for any cylinder $C$, $\mu_n(C)\to 0$ as $n\to \infty$). This quantity measures the complexity at infinity of the system and yields a great deal of information of it. Existence of equilibrium measures as well as phase transitions can be deduced from a good understanding of this quantity. Using Propositions~\ref{prop:tang} and \ref{lem_eq}, the fact that in the sectorially arranged setting, there is no entropy on the boundary also leads to the following immediate conclusion: the entropy map  of the compactified system is upper semicontinuous if and only if $h_\infty=0$. This can also be deduced from \cite[Theorem 1.1]{IomTodVel19}.
\end{remark}

The following result proves that the examples studied by Zargaryan \cite{Zar86} are all sectorially arranged. 

\begin{lemma}
Let $(\Sigma ,\sigma)$ be a finite entropy CMS and $\rho$ a metric of vanishing type on $V$, then the set of vertices is sectorially arranged.
\end{lemma}

\begin{proof}
We will use throughout the fact that $\diam\{n, n+1, \ldots\}\to 0$ in a vanishing type metric so our sectors are always shrinking.
Let $N\in \N$.  Then $\{N+1, N+2, \ldots\}$ can be split into at most countably many disjoint sectors (i.e. connected components of $V$).  Let $V_1^1, V_1^2, \ldots$ be the infinite sectors and choose $N_1$ large enough to cover all the finite sectors, so that  
$$V= \{1, \ldots, N_1\}\cup\bigsqcup_{i=1}^{p_1} V_1^i.$$ 
Given $N'>N_1$ we repeat this procedure obtaining $V= \{1, \ldots, N_2\}\cup\bigsqcup_{i=1}^{p_2} V_2^i.$  Suppose that $V_2^{i}\cap \bigsqcup_{j=1}^{p_1} V_1^j = \es$.  Then there are infinitely many vertices outside $\bigsqcup_{j=1}^{p_1} V_1^j$, a contradiction, so $V_2^{i}\cap V_1^{j_i} \neq \es$ for some $j_i$.  The disjointness of $\{V_1^j\}$ implies that in fact this $j_i$ must be unique with this property, so these sets are nested, as required.
\end{proof}

We conclude this section with the following conjecture.

\begin{conjecture} 
Let $(\Sigma ,\sigma)$ be a finite entropy CMS, $\rho$ totally bounded and $d$ the corresponding metric in $\Sigma$.
Then $P_\Sigma(\phi)=P_{\bar\Sigma}(\phi)$ for all $\phi\in C(\bar\Sigma, \R)$ .
\end{conjecture}

A related question regarding the entropy of locally compact CMS and some of its metric compactifications was posed by Fiebig and Fiebig in  \cite{ff1}. They ask whether there are natural metrics for which the compactification has entropy larger than the original system. We conjecture 
that this never happens for the class of metrics we consider. 

\section{Examples} 

If we endow  $(\Sigma, \sigma)$ with the a metric $d_{\rho, \theta}$ with $\rho$ of vanishing type, then any function $\phi\in UC_{d_{\rho, \theta}}(\Sigma)$ must converge to a unique value on $\bd\Sigma$.  For some systems this is a strong restriction.  In this section, we allow $\rho$ to be non-vanishing, which can enrich  $\bd\Sigma$ and expand  $UC_{d_{\rho, \theta}}(\Sigma)$, and where we still satisfy the conclusions of Corollary~\ref{cor:CMS gat}.

\subsection{Multiple infinities}

We next give a simple sectorially arranged example where we can take a metric $d_{\theta, \rho}$ with $\rho$ of non-vanishing type to enlarge $UC_d(\Sigma)$, but still retain the vanishing type theory.

\begin{example}[Double renewal Shift: two infinities]
We expand the renewal shift, where, for notational convenience we replace the alphabet $\N$ with $\Z$.  Define the $\Z\times\Z$ transition matrix $A=(a_{i,j})$ by $a_{0,j}=1$ for all $j\in \Z$, $a_{i, i-1}=1$ for all $i\ge 1$,  $a_{i, i+1}=1$ for all $i\le -1$, and  $a_{i, j}=0$ otherwise. Define the metric $d_{\rho, \theta}$ as in \eqref{eq:metric} with $\rho:\Z\times\Z\to [0, 1]$ a metric given, for $a, b\in \Z$, by
\begin{equation*}
\rho(a, b) = \begin{cases}0 & \text{ if } a=b,\\
 1 & \text{ if } ab\le 0 \text{ and } a\neq b,\\
\left|\frac1a-\frac1b\right| & \text{ if } ab> 0 \text{ and } a\neq b.\\
 \end{cases}
 \end{equation*}
If we restrict to $\Z_+$ or $\Z_-$ then $\rho$ is actually of vanishing type and our system is just the renewal shift.  Thus we see that $\bar\Sigma$ is obtained in this case by adding `two infinities' to the alphabet: $\bd \Z_\rho=\{-\infty, +\infty\}$.  Clearly this is sectorially arranged.  So, in contrast to the vanishing type case,  $\phi\in UC_{d_{\rho,\theta}}(\Sigma)$ can take different values at $-\infty$ and $+\infty$. 
\end{example}

Clearly we can adapt this example to have $n\in\N$ `infinities' $\bd\Sigma= \{\infty_1, \ldots, \infty_n\}$ in the alphabet corresponding to the compactification.  Moreover,  we can set this up so that to have countable number of `infinities': $\{\infty_1, \infty_2, \ldots\}$. To keep total boundedness, we need these to converge to some limit symbol $\infty_\infty$. 

\subsection{A complex boundary} \label{ssec:comp}

To see that there can be complicated topology  on the boundary of a dynamical system, we will define a CMS with a boundary which is a Cantor set.  This can be obtained from the full shift on three symbols with a full shift on two symbols removed (\cite[Section 6]{IomTod13}), which one can alternatively think of as the Young tower built over the first return map to a 1-cylinder.  Note, however, that as in Lemma~\ref{lem:no_ent}, the dynamics on this boundary are trivial.

To fix notation let $\hat\sigma:\{1, 3\}^{\N_0} \to \{1, 3\}^{\N_0}$ be the usual shift map on this space, though we actually use the standard extension of this to finite words $\hat\sigma:\bigcup_{n\ge 1} \{1, 3\}^n \to \{\epsilon\}\cup\bigcup_{n\ge 1} \{1, 3\}^n$ where $\epsilon$ is the empty word.  We use standard concatenation notation here, where in particular for any word $\underline w$, $\epsilon \underline w=\underline w=\underline w\epsilon$. Now in our example, the alphabet is 
$$\Sigma:=\left\{<\underline w2>: \underline w \in  \{\epsilon\}\cup\bigcup_{n\ge 1}\{1, 3\}^n\right\}$$
with transitions $2\to <\underline w2>$ allowed for any $\underline w\in \{\epsilon\}\cup\bigcup_{n\ge 1}\{1, 3\}^n$ and otherwise $<\underline w2>\to <\underline w'2>$ allowed only if $\underline w'=\hat\sigma \underline w$. The shift $\sigma:\Sigma\to \Sigma$ is the usual left-shift.  The renewal-like structure here means that each point $\underline x\in \Sigma$ must be of the form 
\begin{equation}
\underline x=\left(<\underline w^12>, <\hat\sigma \underline w^12>, \ldots, <\sigma^{n_1}\underline w^12>, <\underline w^22>, <\hat\sigma \underline w^22>, \ldots\right)
\label{eq:3shift}
\end{equation}
where $\underline w^i \in  \bigcup_{n\ge 1}\{1, 3\}^n$ and $n_i$ is $|\underline w^i|$ the length of $\underline w^i$.

It may be convenient for the reader to view this system as a dynamical system on the dyadic tree, where the action of the dynamics is to send the root everywhere and then each other vertex is sent only to the adjacent vertex which is one step closer to the root (see Figure \ref{fig:tree}).

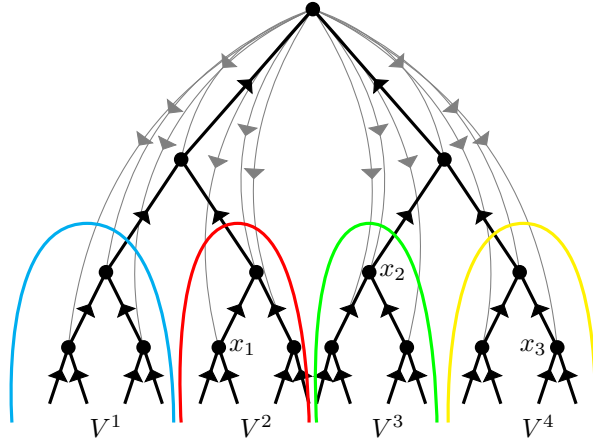
\begin{figure}[h]
\begin{tikzpicture}[thick, scale=0.5]
    \begin{scope}[very thick,nodes={sloped,allow upside down}]

\draw [thin, gray] (10,20) to[bend right] node {\midarrowdown} (3.5,11);

\draw [thin, gray] (10,20) to[bend right=40] node {\midarrowdown} (5.5,11);

\draw [thin, gray] (10,20) to[bend right] node {\midarrowdown} (7.5,11);

\draw [thin, gray] (10,20) to[bend right] node {\midarrowdown} (9.5,11);

\draw [thin, gray] (10,20) to[bend right] node {\midarrowdown} (4.5,13);

\draw [thin, gray] (10,20) to[bend right] node {\midarrowdown} (8.5,13);

\draw [thin, gray] (10,20) to[bend right] node {\midarrowdown} (6.5,16);

\draw [thin, gray] (10,20) to[bend left] node {\midarrowdown} (10.5,11);

\draw [thin, gray] (10,20) to[bend left] node {\midarrowdown} (12.5,11);

\draw [thin, gray] (10,20) to[bend left=40] node {\midarrowdown} (14.5,11);

\draw [thin, gray] (10,20) to[bend left] node {\midarrowdown} (16.5,11);

\draw [thin, gray] (10,20) to[bend left] node {\midarrowdown} (11.5,13);

\draw [thin, gray] (10,20) to[bend left] node {\midarrowdown} (15.5,13);

\draw [thin, gray] (10,20) to[bend left] node {\midarrowdown} (13.5,16);

\draw (10,20) -- node {\midarrow} (6.5,16);
\fill (10,20) circle (5.5pt);
\draw (10,20) -- node {\midarrow} (13.5,16);

\draw (6.5,16) -- node {\midarrow} (4.5,13);
\fill (6.5,16) circle (5.5pt);
\draw (6.5,16) -- node {\midarrow} (8.5,13);

\draw (13.5,16) -- node {\midarrow} (11.5,13);
\fill (13.5,16) circle (5.5pt);
\draw (13.5,16) -- node {\midarrow} (15.5,13);

\draw (4.5,13) -- node {\midarrow} (3.5,11);
\fill (4.5,13) circle (5.5pt);
\draw (4.5,13) -- node {\midarrow} (5.5,11);

\draw (8.5,13) -- node {\midarrow} (7.5,11);
\fill (8.5,13) circle (5.5pt);
\draw (8.5,13) -- node {\midarrow} (9.5,11);

\draw (11.5,13) -- node {\midarrow} (10.5,11);
\fill  (11.5,13) circle (5.5pt);
\node[right] at (11.5, 13) {$x_2$};
\draw (11.5,13) -- node {\midarrow} (12.5,11);

\draw (15.5,13) -- node {\midarrow} (14.5,11);
\fill (15.5,13) circle (5.5pt);
\draw (15.5,13) -- node {\midarrow} (16.5,11);

\draw (5.5,11) -- node {\midarrow} (5.0,9.5);
\fill (5.5,11) circle (5.5pt);
\draw (5.5,11) -- node {\midarrow} (6.0,9.5);

\draw (3.5,11) -- node {\midarrow} (3.0,9.5);
\fill (3.5,11) circle (5.5pt);
\draw (3.5,11) -- node {\midarrow} (4.0,9.5);

\draw (7.5,11) -- node {\midarrow} (7,9.5);
\fill (7.5,11) circle (5.5pt);
\node[right] at (7.5, 11) {$x_1$};
\draw (7.5,11) -- node {\midarrow} (8,9.5);

\draw (9.5,11) -- node {\midarrow} (9,9.5);
\fill (9.5,11) circle (5.5pt);
\draw (9.5,11) -- node {\midarrow} (9.9,9.5);

\draw (10.5,11) -- node {\midarrow} (10.1,9.5);
\fill(10.5,11) circle (5.5pt);
\draw (10.5,11) -- node {\midarrow} (11.0,9.5);

\draw (12.5,11) -- node {\midarrow} (12,9.5);
\fill (12.5,11) circle (5.5pt);
\draw (12.5,11) -- node {\midarrow} (13,9.5);

\draw (14.5,11) -- node {\midarrow} (14,9.5);
\fill (14.5,11) circle (5.5pt);
\draw (14.5,11) -- node {\midarrow} (15,9.5);

\draw (16.5,11) -- node {\midarrow} (16.0,9.5);
\fill (16.5,11) circle (5.5pt);
\node[left] at (16.5, 11) {$x_3$};
\draw (16.5,11) -- node {\midarrow} (17.0,9.5);

\draw [cyan] plot [smooth, tension=2] coordinates { (2,9) (4,14.3) (6.3, 9)};

\draw [red] plot [smooth, tension=2] coordinates { (6.5,9) (8,14.3) (9.9, 9)};

\draw [green] plot [smooth, tension=2] coordinates { (10.1,9) (11.5,14.3) (13.3, 9)};

\draw [yellow] plot [smooth, tension=2] coordinates { (13.6,9) (15.6,14.3) (17.5, 9)};

\node[] at (4.5, 8.9) {$V^1$};

\node[] at (8.5, 8.9) {$V^2$};

\node[] at (12, 8.9) {$V^3$};

\node[] at (16, 8.9) {$V^4$};

\end{scope}

\end{tikzpicture}

\caption{An example of sectors with a complex boundary with a renewal-like structure: the vertices are dots, the arrows given the Markov structure.  The diagram continues downwards with infinitely many vertices and arrows in the same pattern.  The arrows going from the base vertex are in grey so as not to obscure the structure too much.  We pick out particular nodes $x_1, x_2, x_3$ to show how the metric works in the text.}  \label{fig:tree}
\end{figure}

We define the metric $\rho$ by 
\begin{equation}
\rho(<\underline w2>, <\underline w'2>) = \begin{cases} 0 & \text{ if } \underline w=\underline w',\\
\frac1{1+\min\left\{i:w_{|w|-i}\neq w_{|w'|-i}'\right\}} & \text{ otherwise}.\\
 \end{cases} \label{eq:tree_rho}
 \end{equation}
 (Note that here if $|\underline w|\neq |\underline w'|$ then we make up the difference by appending the appropriate number of $\epsilon$s to the front of the shorter word.) 
Clearly $\rho$ is totally bounded.  Note that the boundary can be viewed as the space of paths in the dyadic tree, or as $\{1, 3\}^{\N_0}$.  The tree structure and the corresponding metric make it easy to see that this is sectorially arranged, see Figure~\ref{fig:tree}.  In that figure, the nodes $x_1=0102$, $x_2=012$, $x_3=1112$.  Thus
$\rho(x_1, x_2)=\frac{1}{1+0}=1, \ \rho(x_1, x_3)=\frac{1}{1+0}=1, \ \rho(x_2, x_3)=\frac{1}{1+1}=\frac12$.  For the examples of sectors there, $V^1$ has all vertices of the form $\ldots 002$,  $V^2$ has all vertices of the form $\ldots102$, $V^3$ has all vertices of the form $\ldots 012$, and $V^4$ has all vertices of the form $\ldots 112$.

The metric we choose here is different to the one we would inherit on $\{1, 3\}^{\N_0}$ from the usual metric on $\{1, 2, 3\}^{\N_0}$.   This primarily due to the fact that we `count backwards', i.e. in \eqref{eq:tree_rho} we take $\min\left\{i:w_{|w|-i}\neq w_{|w'|-i}'\right\}$, but to get back to the usual metric (totally changing the structure here) we would replace this with $\min\left\{i:w_{i}\neq w_{i}'\right\}$.

\begin{remark}
These methods  can be extended to, for example, start with a shift of finite type $\Sigma$ and replace the root vertex $[2]$ with $\F=\{v_1, v_2, \ldots\}$.  Then the boundary corresponds to $\Sigma_{\F}$, a subshift with forbidden words (added to the forbidden words for $\Sigma$).  The dynamics, however, just fixes every point.
\end{remark}

\subsection{Not sectorially arranged, but no new entropy} \label{ssec:nsec}
In these examples we create a complicated boundary via a particular representation of the renewal shift, but as we will see this does not lead to more entropy.  Let $S:[0, \infty) \to [0, 1]\times [-1, 0]$ be a continuous curve which can be written as a countable union $\{S_k\}_{k \in \N}$,  where each $S_k:[k, k+1]\to [0, 1]\times [-1, 0]$ is a straight line, parametrised with constant speed. Denote by $\text{Proj}_x$ the projection on the $x$ coordinate. We assume that  $\text{Proj}_x(S_k(k))= 0, \text{Proj}_x(S_k(k+1))=1$, if $k$ is even and  $\text{Proj}_x(S_k(k))= 1, \text{Proj}_x(S_k(k+1))=0$ if $k$ is odd. In both cases assume the $y$ coordinate of $S_k$ increases, and moreover assume that $S$ accumulates on $[0, 1]\times \{0\}$. This is a zig-zag pattern accumulating on a line. Thinking of $V$ as the vertices of the renewal shift and the metric $\rho$ as coming from the placement of $V$ through $S$, in the Euclidean metric on $ [0, 1]\times [-1, 0]$, 
we can create various examples.  The first simple, non-trivial, example is to put the vertex $2k$ at $S_k(k)$ and $2k+1$ at $S_k(k+1)$.  The arrows will go either from 0 to all vertices, or  `downhill', as in the standard renewal shift model (see Figure \ref{fig:per2}).  Then $\bd V=\bd V_\rho$ consists of two points $x_{\infty, 1}, x_{\infty, 2}$, and it can be seen from the construction that aside from the usual renewal shift transitions, $x_{\infty, 1} \to  x_{\infty, 2}$, $x_{\infty, 2}\to x_{\infty, 1}$ and $0 \to x_{\infty, 1}$, $0\to x_{\infty, 2}$ are allowed.

\begin{figure}[h]
\begin{tikzpicture}[thick, scale=0.5]
    \begin{scope}[very thick,nodes={sloped,allow upside down}]

\draw (0, -10)  -- node {\midarrow} (10, -8) ;
\fill (0, -10) circle (5.5pt);

\draw  (10, -8) -- node {\midarrow} (0, -6.5) ;
\fill (10, -8) circle (5.5pt);    

\draw  (0, -6.5) -- node {\midarrow} (10, -5.5) ;
\fill (0, -6.5) circle (5.5pt);    

\draw  (10, -5.5) -- node {\midarrow} (0, -4.7) ;
\fill (10, -5.5) circle (5.5pt);

\draw  (0, -4.7) -- node {\midarrow} (10, -3.9) ;
\fill (0, -4.7) circle (5.5pt);

\draw  (10, -3.9) -- node {\midarrow} (0, -3.2) ;
\fill (10, -3.9) circle (5.5pt);        
        
\draw  (0, -3.2) -- node {\midarrow} (10, -2.7) ;
\fill (0, -3.2) circle (5.5pt);    
\fill (10, -2.7) circle (5.5pt);

\draw  (0, -0.5) to[bend left=10] node {\midarrow} (10, -0.5) ;
\draw  (0, -0.5) to[bend right=10] node {\midarrowdown} (10, -0.5) ;

\fill (0, -0.5) circle (5.5pt);    
\fill (10, -0.5) circle (5.5pt); 

\node[] at (0, 0) {$x_{\infty, 1}$};
\node[] at (10, 0) {$x_{\infty, 2}$};

\draw[dotted, thin] (2, -1.4)--(2, -2.5);

\draw[dotted, thin] (4, -1.4)--(4, -2.5);

\draw[dotted, thin] (6, -1.4)--(6, -2.5);
\draw[dotted, thin] (8, -1.4)--(8, -2.5);

\draw [thin, gray] (0, -10)   --node {\midarrowdown} (0,-6.5);

\draw [thin, gray] (0, -10) to[bend left]  node {\midarrowdown} (0,-4.7);

\draw [thin, gray] (0, -10) to[bend left]  node {\midarrowdown} (0,-3.2);

\draw [thin, gray] (0, -10) to[bend left]  node {\midarrowdown} (0,-0.5);

\draw [thin, gray] (0, -10) to[bend right=20] node {\midarrowdown} (10,-8);

\draw [thin, gray] (0, -10)  -- node {\midarrowdown} (10,-5.5);

\draw [thin, gray] (0, -10)  -- node {\midarrowdown} (10,-3.9);

\draw [thin, gray] (0, -10)  -- node {\midarrowdown} (10,-2.7);

\draw [thin, gray] (0, -10)  -- node {\midarrowdown} (10, -0.5);

\node[] at (-0.4, -10) {$0$};
\node[] at (-0.4, -6.5) {$2$};
\node[] at (-0.4, -4.7) {$4$};
\node[] at (-0.4, -3.2) {$6$};

\node[] at (10.4, -8) {$1$};
\node[] at (10.4, -5.5) {$3$};
\node[] at (10.4, -3.9) {$5$};
\node[] at (10.4, -2.7) {$7$};

\end{scope}

\end{tikzpicture}

\caption{Zero entropy example with a renewal structure which is not sectorially arranged.  The boundary is given here also.}  \label{fig:per2}\end{figure}
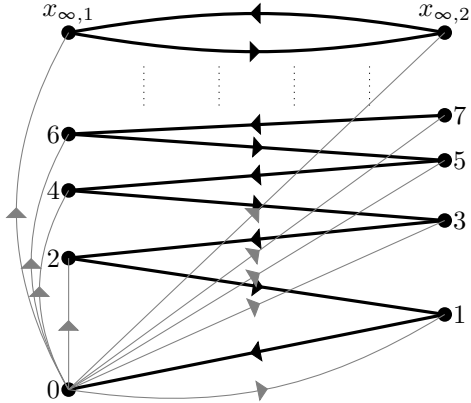

We can also create a system with three boundary points $x_{\infty, 1},  x_{\infty, 2}, x_{\infty, 3}$ which can each transition to all the others.  If the dynamics on the boundary were Markov then this would have positive entropy, but in fact the dynamics turn out to be highly deterministic.
We do this by choosing a pattern of finite zig-zags, here on 4 curves, and repeating it periodically so that it accumulates periodically on $[0, 1]\times \{0\}$.  One can think of $x_{\infty, 1}=0$, $x_{\infty, 2}=1/2$ and $x_{\infty, 3}=1$.  On the first two curves we connect the point projecting to 0 with the point projecting to 1 and vice versa.  In the next two, we connect the point projecting to 0 with the point projecting to $1/2$ with the point projecting to 1 and vice versa.  (Note that since this is renewal shift, we `connect backwards'.) With this pattern established, in the limit we obtain the boundary we claimed.  However the dynamics on the boundary is periodic, one representative given by $(\overline{x_{\infty, 1}, x_{\infty, 2}, x_{\infty, 3}, x_{\infty, 2}, x_{\infty, 1}, x_{\infty, 3}})$.

This type of example can be made to include arbitrarily many points on the boundary, but still with periodic behaviour.

\subsection{The boundary being the circle} In Section \ref{ssec:nsec} we constructed a metric $\rho$ in $V$ from the Euclidean metric in the plane. This is a rather flexible technique that allows for the construction of systems with different types of boundaries.  As an example, we construct a CMS having the unit circle as boundary. Let $p(n)$ be a strictly increasing sequence of positive integers with $p(1)=1$. The graph defining the CMS is composed of $p(n)$ loops of length $n$ based at the vertex $v$, for every $n \in \N$. Consider the set of vertices 
\begin{equation*}
V=\left\{ v \right\} \cup \bigcup_{n=2}^{\infty} \bigcup_{k=1}^{p(n)} \left\{ b_{nk}(1), \dots, b_{nk}(n-1)		\right\}.
\end{equation*}
The non zero entries of the transition matrix $A=(a_{i,j})$  are exactly $a_{v,v}$ and $a_{v, b_{nk}(1)}, a_{b_{nk}(i), b_{nk}(i+1)}, a_{b_{nk}(n-1), v}$, for all $n,k \geq 1$. This class of systems is usually called \emph{loop systems} and have been studied by several authors, see for example \cite[Example 29]{rue} and \cite[Section 5]{s2}. We will consider an increasing sequence of circles and on each of them equidistribute the vertices of a single loop. Let $(r_n)_n$ be  a strictly increasing sequence of positive real numbers  converging to $1$. Consider a sequence of circles of radii $r_n$ each of which is centred at $(0,0)$. Place the vertex $\{v\}$ at $(0,0)$ and equidistribute the vertices of each loop $\left\{ b_{nk}(1), \dots, b_{nk}(n-1)		\right\}$ in a circle, starting from that of radii $r_1$ following the numeration of $V$.  Let $\rho$ be the  metric on the vertices induced by the Euclidean metric. The boundary $\partial V$ is $S^1$ and the dynamics extends to it, fixing each point.

\subsection{Positive entropy interior rich example}\label{ssec:tangnoteq}
The examples in the previous subsection, as well as for example the renewal shift, are interior rich, but have zero entropy at the boundary.
The following example is not sectorially arranged, but is interior rich with positive entropy on the boundary.

We define the one-sided birth and death chain on $\N$ with an unusual boundary.  That is,  for $a, b\in \N$, if $a>1$ then $a\to b$ if $b\in \{a-1, a, a+1\}$; and if $a=1$, $a\to b$ if $b\in \{1, 2\}$.  Now define
\begin{equation*}
\rho(a, b)=\begin{cases} 1 &\text{ if } a+b \text{ is odd},\\
\left|\frac1a-\frac1b\right| &\text{ if } a+b \text{ is even}.
\end{cases}
\end{equation*}

Therefore,  $\bd \N_\rho=\{\infty_o,\infty_e\}$ where all transitions between these elements are allowed, so $h_{top}(\bd \Sigma_\rho)=\log 2$.  However, the dynamics on the boundary, and hence the measures, are completely mirrored in $\Sigma$ arbitrarily close to the boundary since the full shift on $\{n, n+1\}$ is a subset of $\Sigma$.  Thus this example is interior rich.  It is also easy to see that this cannot be sectorially arranged.  One can show that $h_{top}(\Sigma)=\log 3$.    Finally, recalling Proposition~\ref{prop:tang}, we observe that for the potential $\phi\equiv 0$, the measure of maximal entropy on $\bd\Sigma_\rho$ is a tangent functional, but since it has entropy strictly less than $h_{top}(\Sigma)$, it is not an equilibrium measure.


\begin{thebibliography}{XXX}


\bibitem[A]{a} N.\ Aronszajn \emph{Differentiability of Lipschitzian mappings between Banach spaces.} Studia Math. 57 (1976), no. 2, 147--190.


\bibitem[BS]{bs} L.\ Barreira and J.\  Schmeling \emph{Sets of "non-typical'' points have full topological entropy and full Hausdorff dimension. } Israel J. Math. 116 (2000), 29--70.


\bibitem[BCL]{bcl}  F.\ B\'eguin, S.\ Crovisier and F.\ Le Roux \emph{Construction of curious minimal uniquely ergodic homeomorphisms on manifolds: the Denjoy-Rees technique.} Ann. Sci. \'Ecole Norm. Sup. (4) 40 (2007), no. 2, 251--308.

\bibitem[BL]{bl} Y.\ Benyamini and  J.\ Lindenstrauss  \emph{Geometric nonlinear functional analysis. Vol. 1. }
American Mathematical Society Colloquium Publications, 48. American Mathematical Society, Providence, RI, 2000.


\bibitem[BCMV]{bcmv}   A.\ Bis, M.\ Carvalho, M.\ Mendes and  P.\ Varandas \emph{A convex analysis approach to entropy functions, variational principles and equilibrium states.} Comm. Math. Phys. 394 (2022), 215--256.


\bibitem[Bo1]{bow1} R.\ Bowen \emph{Topological entropy for noncompact sets.} Trans. Amer. Math. Soc. 184 (1973), 125--136.

\bibitem[Bo2]{bow}
R.\ Bowen \emph{Equilibrium states and the ergodic theory of Anosov diffeomorphisms.} Second revised edition.  With a preface by David Ruelle.  Edited by Jean-Ren\'e Chazottes.  Lecture Notes in Mathematics, 470. Springer-Verlag, Berlin 2008.


\bibitem[C]{c} M.\ Cs\"ornyei  \emph{Aronszajn null and Gaussian null sets coincide.} Israel J. Math. 111 (1999), 191--201.


\bibitem[DvE]{dv} H.A.M. Dani\"els and A.C.D. van Enter \emph{Differentiability properties of the pressure in lattice systems.} Comm. Math. Phys. 71 (1980), no. 1, 65--76.


\bibitem[D]{Dobr} R.L.\ Dobru\v sin
\emph{Description of a random field by means of conditional probabilities and conditions for its regularity.}
Teor. Verojatnost. i Primenen 13 (1968) 201--229.


\bibitem[FF1]{ff1} D.\ Fiebig and U.\ Fiebig \emph{Topological boundaries for countable state Markov shifts.}
Proc. London Math. Soc. (3) 70 (1995), no. 3, 625--643.

\bibitem[F]{f} D.\ Fiebig  \emph{Canonical compactifications for Markov shifts.} Ergodic Theory Dynam. Systems 33 (2013), no. 2, 441--454.

\bibitem[Gu1]{gu1} B.M.\ Gurevi\v c
\emph{Topological entropy for denumerable Markov chains,}
Dokl. Akad. Nauk SSSR  10 (1969) 911--915.

\bibitem[Gu2]{gu2} B.M.\ Gurevi\v c
\emph{Shift entropy and Markov measures in the path space of a denumerable graph,}
Dokl. Akad. Nauk SSSR 11 (1970) 744--747.


\bibitem[GS]{gs} B.M.\ Gurevich and S.V.\ Savchenko \emph{Thermodynamic formalism for symbolic Markov chains with a countable number of states.} Uspekhi Mat. Nauk 53 (1998), no. 2(320), 3--106; translation in Russian Math. Surveys 53 (1998), no. 2, 245--344

\bibitem[HK]{hk} M.\ Handel and B.\ Kitchens \emph{Metrics and entropy for non-compact spaces.} With an appendix by Daniel J. Rudolph. Israel J. Math. 91 (1995), no. 1-3, 253--271.

\bibitem[H]{h} F.\ Hofbauer, \emph{Examples for the nonuniqueness of the equilibrium state.} Trans. Amer. Math. Soc. 228 (1977), 223--241

\bibitem[IT]{IomTod13}  G.\ Iommi and M.\ Todd \emph{Transience in dynamical systems.} Ergodic Theory Dynam. Systems, 33 (2013) 1450--1476.

\bibitem[ITV]{IomTodVel19}  G.\ Iommi, M.\ Todd and A.\ Velozo \emph{Escape of entropy for countable Markov shifts.} Adv. Math. 405 (2022), Paper No. 108507.

\bibitem[IP]{ip}   R.B.\ Israel and R.R.\ Phelps  \emph{Some convexity questions arising in statistical mechanics.} Math. Scand. 54 (1984), no. 1, 133--156.

\bibitem[L]{l} Y.\ Lima  \emph{Symbolic dynamics for non-uniformly hyperbolic systems.} Ergodic Theory Dynam. Systems 41 (2021), no. 9, 2591--2658.

\bibitem[MU]{mu} R.D.\ Mauldin and  M.\ Urba\'nski \emph{Dimensions and measures in infinite iterated function systems.} Proc. London Math. Soc. (3) 73 (1996), no. 1, 105--154.

\bibitem[PP]{PesPit} Y.B.\ Pesin and B.S.\ Pitskel'
\emph{Topological pressure and the variational principle for noncompact sets,} Funktsional. Anal. i Prilozhen. \textbf{18} (1984) 50--63. 

\bibitem[Pe]{pe} Y.B.\ Pesin  \emph{Dimension theory in dynamical systems. Contemporary views and applications.} Chicago Lectures in Mathematics. University of Chicago Press, Chicago, IL, 1997. xii+304 pp.

\bibitem[PS]{ps} C.E.\ Pfister and  W.G.\ Sullivan  \emph{Large deviations estimates for dynamical systems without the specification property. Applications to the $\beta$-shifts.} Nonlinearity 18 (2005), no. 1, 237--261.


\bibitem[P]{p3} R.R.\ Phelps \emph{Convex functions, monotone operators and differentiability.} Second edition. Lecture Notes in Mathematics, 1364. Springer-Verlag, Berlin, 1993. xii+117 pp.

\bibitem[Re]{Ree81} M.\ Rees 
\emph{A minimal positive entropy homeomorphism of the 2-torus.}
 J.\ London Math.\ Soc.\ (2) 23 (1981) 537--550. 

\bibitem[Ru]{ru}  D.\ Ruelle \emph{Thermodynamic formalism. The mathematical structures of classical equilibrium statistical mechanics.} With a foreword by Giovanni Gallavotti and Gian-Carlo Rota. Encyclopedia of Mathematics and its Applications, 5. Addison-Wesley Publishing Co., Reading, Mass., 1978.

\bibitem[Rue]{rue} S.\ Ruette \emph{On the Vere-Jones classification and existence of maximal measures for countable topological Markov chains.} Pacific J. Math. 209 (2003), no. 2, 366--380.


\bibitem[Ro]{r} R.T.\ Rockafellar  Convex analysis. Princeton Mathematical Series, No. 28 Princeton University Press, Princeton, N.J. 1970 xviii+451 pp.


\bibitem[S1]{Sar99}  O.\ Sarig \emph{Thermodynamic formalism for countable Markov shifts.} Ergodic Theory Dynam. Systems 19 (1999), no. 6, 1565--1593

\bibitem[S2]{s2} O.\ Sarig \emph{Phase transitions for countable Markov shifts.} Comm. Math. Phys. 217 (2001), no. 3, 555--577. 


\bibitem[S3]{s3} O.\ Sarig \emph{Symbolic dynamics for surface diffeomorphisms with positive entropy.} J. Amer. Math. Soc. 26 (2013), no. 2, 341--426.

\bibitem[Sh]{sh} O.\ Shwartz \emph{Thermodynamic formalism for transient potential functions.} Comm. Math. Phys. 366 (2019), no. 2, 737--779. 

\bibitem[Si]{Sinai} J.G.\ Sinai
\emph{Gibbs measures in ergodic theory.}
Uspehi Mat. Nauk 27 (1972) 21--64.

\bibitem[T]{Tho11}  D.J.\ Thompson,
 \emph{A thermodynamic definition of topological pressure for non-compact sets.}
  Ergodic Theory Dynam. Systems \textbf{31} (2011) 527--547.

\bibitem[W1]{Wal78}   P.\ Walters,    \emph{Invariant measures and equilibrium states for some mappings which expand distances. } Trans. Amer. Math. Soc. 236 (1978), 121--153.

\bibitem[W2]{Wal82}  P.\ Walters  \emph{An introduction to ergodic theory.} Graduate Texts in Mathematics, 79. Springer-Verlag, New York-Berlin, 1982. ix+250 pp.


\bibitem[W3]{Wal92} P.\  Walters  \emph{Differentiability properties of the pressure of a continuous transformation on a compact metric space.} J. London Math. Soc. (2) 46 (1992), no. 3, 471--481


\bibitem[Z]{Zar86}  A.S.\ Zargaryan \emph{A variational principle for the topological pressure in the case of a Markov chain with a countable number of states.} Mat. Zametki 40 (1986), no. 6, 749--761, 829.

\end{thebibliography}
\end{document}